\documentclass{article}
\usepackage[utf8]{inputenc}
\usepackage{amssymb}
\usepackage{amsmath}
\usepackage{amsthm}
\usepackage{bm}
\usepackage[dvipsnames]{xcolor}
\usepackage{tikz}
\usepackage{chngcntr}

\newtheorem{theorem}{Theorem}[section]
\newtheorem{corollary}[theorem]{Corollary}
\newtheorem{lemma}[theorem]{Lemma}
\newtheorem{definition}[theorem]{Definition}
\newtheorem{proposition}[theorem]{Proposition}

\newtheorem*{remark}{Remark}

\newtheorem{question}{Question}[section]
\counterwithin{figure}{section}

\title{Intrinsic Smallness}
\author{Justin Miller\footnote{The author would like to thank his advisor, Dr. Peter Cholak, for the advice, discussion, and support that made this project possible.}\footnote{Partially supported by NSF-DMS-1854136}}
\date{}

\begin{document}

\maketitle

\begin{abstract}
    Recent work in computability theory has focused on various notions of asymptotic computability, which capture the idea of a set being ``almost computable.'' One potentially upsetting result is that all four notions of asymptotic computability admit ``almost computable'' sets in every Turing degree via coding tricks, contradicting the notion that ``almost computable'' sets should be computationally close to the computable sets. In response, Astor introduced the notion of intrinsic density: a set has defined intrinsic density if its image under any computable permutation has the same asymptotic density. Furthermore, introduced various notions of intrinsic computation in which the standard coding tricks cannot be used to embed intrinsically computable sets in every Turing degree. Our goal is to study the sets which are intrinsically small, i.e. those that have intrinsic density zero. We begin by studying which computable functions preserve intrinsic smallness. We also show that intrinsic smallness and hyperimmunity are computationally independent notions of smallness, i.e. any hyperimmune degree contains a Turing-equivalent hyperimmune set which is ``as large as possible'' and therefore not intrinsically small. Our discussion concludes by relativizing the notion of intrinsic smallness and discussing intrinsic computability as it relates to our study of intrinsic smallness.
\end{abstract}


\section{Introduction}
A noteworthy phenomenon in the world of computing is that of problems which are generally ``easy'' to compute but have very difficult worst case instances. This gave rise to the notion of \textit{generic computability}, studied by Kapovich, Myasnikov, Schupp, and Shpilrain \cite{generic} in the context of computing the word problems of finitely generated groups. This notion asserts that a set is computable outside of a ``small'' error set where the algorithm does not answer. The notion of smallness here is that of having asymptotic density $0$:

\begin{definition}
    The partial density of $A\subseteq \omega$ at $n$ is 
    \[\rho_n(A)=\frac{|A\upharpoonright n|}{n}.\] 
    That is, it is the ratio of the number of things less than $n$ that are in $A$ to what could be in $A$. The upper (asymptotic) density of $A$ is 
    \[\overline{\rho}(A)=\limsup_{n\to\infty} \rho_n(A)\]
    and the lower (asymptotic) density of $A$ is 
    \[\underline{\rho}(A)=\liminf_{n\to\infty} \rho_n(A).\]
    If $\overline{\rho}(A)=\underline{\rho}(A)$, we call this limit the (asymptotic) density of $A$ and denote it by $\rho(A)$.
\end{definition}

Recall that $W_e$ is the domain of the $e$-th Turing machine $\varphi_e$.

\begin{definition}
    A set $A$ is generically computable if there is a partial computable function $\varphi_e$ such that $\underline{\rho}(W_e)=1$ and if $\varphi_e(n)\downarrow$, then $\varphi_e(n)=A(n)$. $\varphi_e$ is called a generic description of $A$.
\end{definition}

We think of generically computable sets as being computable ``almost everywhere,'' i.e. there is an algorithm that correctly answers questions on a set of density $1$, but does not answer on a small (density $0$) error set. Here the error set is the set of $n$ on which the description diverges. By changing the behavior of the generic description from diverging to something else, we obtain the other three notions of generic computability.

\begin{definition}
    A set $A$ is coarsely computable if there is a total computable function $\varphi_e$ such that $\underline{\rho}(\{n:\varphi_e(n)=A(n)\})=1$. $\varphi_e$ is called a coarse description of $A$.
\end{definition}

For coarse computability, the description is forced to answer every question, but is allowed to give the incorrect answer on the error set. That is, the error set is the set of numbers on which the description and the set disagree.

\begin{definition}
    A set $A$ is densely computable if there is a partial computable function $\varphi_e$ such that $\underline{\rho}(\{n:\varphi_e(n)\downarrow=A(n)\})=1$. $\varphi_e$ is called a dense description of $A$.
\end{definition}

For dense computability, the description can both answer questions incorrectly and not answer them on the error set. More specifically, the error set is both the places where the description diverges and the places where it disagrees with the set.

\begin{definition}
    A set $A$ is effectively densely computable if there is a total computable function $\varphi_e:\omega\to\{0,1,\square\}$ such that $\underline{\rho}(\varphi_e^{-1}(\{0,1\}))=1$ and $\varphi_e(n)\in\{0,1\}$ implies $\varphi_e(n)=A(n)$. ( $\square$ represents $\varphi_e(n)$ refusing to answer whether $n$ is in or out of the set.)
\end{definition}

Effective dense computability need not answer questions on the error set much like generic computability, but it must refuse to do so outright rather than running for infinite time. (That is, the error set, which is the inverse image of $\square$ under the description, must be computable.) Note that there are some immediate implications among these notions. Effective dense computability implies both coarse computability and generic computability, and both of these imply dense computability. For an overview of the history of these notions, refer to the first section of \cite{dense}.
\\
\\
One potentially unsettling feature of all four notions of asymptotic computability is that they depend heavily on the way in which information is coded. In fact, Jockusch and Schupp \cite{coarse} give a simple argument that can show every Turing degree contains a set which is effectively densely computable by ``hiding'' an entire set of any degree on a small computable set such as the factorial. (As the other three notions are implied by effective dense computability, the same is automatically true for every notion of asymptotic computability.) 

\begin{proposition}
Let $X\subseteq \omega$. Then there is $A\equiv_T X$ which is effectively densely computable.
\end{proposition}

\begin{proof}
    Given $X$, let $A=\{n!:n\in X\}$. Then $A$ is clearly Turing equivalent to $X$, and the function
    \[ f(n) = \begin{cases}
            \square & \mbox{if } n=k!\\
            0 &\mbox{otherwise }
        \end{cases}
    \]
    witnesses that $A$ is effectively densely computable.
\end{proof}

Therefore, these notions of being ``almost'' computable are heavily dependent upon how the set is coded: computably re-arranging the elements of a set can break the property of being ``almost computable.'' To combat this, Astor \cite{intrinsicdensity} introduced the notion of \textit{intrinsic density}, a strengthening of asymptotic density. Let $Perm$ be the index set of computable permutations of $\omega$.

\begin{definition}
The absolute upper density of $A\subseteq\omega$ is
\[\overline{P}(A)=\sup\{\overline{\rho}(\pi(A)):\pi\text{ a computable permutation}\}\]
and the absolute lower density of $A$ is
\[\underline{P}(A)=\inf\{\underline{\rho}(\pi(A)):\pi\text{ a computable permutation}\}.\]
If $\overline{P}(A)=\underline{P}(A)$, then we call this limit the intrinsic density of $A$ and denote it by $P(A)$.
\end{definition}

(In particular, if $A$ has intrinsic density $0$, then $\overline{\rho}(\pi(A))=0$ for every computable permutation. Furthermore, $\overline{P}(A)=0$ is enough to ensure $A$ has intrinsic density zero.) Of special interest is the property of having intrinsic density $0$, which has been studied extensively by Astor \cite{intrinsicdensity},\cite{intrinsicsmallness} in relation with other notions of smallness such as immunity. We will refer to sets that have intrinsic density $0$ as \textit{intrinsically small} to ease notation slightly. Technically finite sets meet this definition, but from here on we shall use the term to refer to infinite sets as those are the interesting ones.) We wish to study intrinsically small sets in order to use them as our error sets in an intrinsic version of asymptotic computability which we shall discuss in Section 5.
\\
\\
One easy observation about intrinsically small sets is that there are more computable functions $f$ such that $\overline{\rho}(f(A))=0$ for all intrinsically small sets $A$ than just the computable permutations. For example, if $\pi$ is a computable permutation, then $2\cdot \pi$ is not a computable permutation but the image of any intrinsically small set under it still has density $0$. The following definition captures the idea of classes of functions preserving smallness.

\begin{definition}
   For a class $\mathcal{F}$ of (partial) computable functions from $\omega$ to $\omega$, we say that $A\subset\omega$ is \textit{small for} $\mathcal{F}$ if $\overline{\rho}(f(A))=0$ for every $f\in\mathcal{F}$. 
\end{definition}

Notice that $A$ is intrinsically small if and only if it is small for computable permutations. In Section 2, we shall explore which classes of functions $\mathcal{F}$ have the property that every intrinsically small set is small for $\mathcal{F}$. This will give rise to a few questions, which we will study further in Section 3. In Section 4 we shall describe and explore the relativization of intrinsic smallness.

\section{Functions and Intrinsic Density}

We first note that not all intrinsically small sets are small for all computable functions, nor even all total computable functions. To do so, we use the following lemma:

\begin{lemma}
\label{jumpstrategy}
Let $X$ be a set of natural numbers. Suppose that $\{\mathcal{R}_e\}_{e\in\omega}$ is a collection of uniformly $X$-computable infinite sets. Then there is an intrinsically small set $A\leq\emptyset'\oplus X$ such that $A\cap \mathcal{R}_e\neq\emptyset$ for all $e$.
\end{lemma}

\begin{proof}
Note that the index set of injective partial computable functions is $\emptyset'$ computable, as the index set of noninjective partial computable functions is $\Sigma_1^0$. Therefore there is a $\emptyset'$-computable function $f$ such that $\varphi_{f(e)}$ is an enumeration of exactly the injective partial computable functions.
\\
\\
Let $A_0=\emptyset$ and $r_0=0$. Given $A_s$, $R_s$, define $A_{s+1},r_{s+1}$ as follows: Using $X$ as an oracle, find $k$ the least element of $R_s$ with $k>r_{s+1}$, which exists because $R_s$ is infinite. Let $A_{s+1}=A_s\cup\{k\}$. We say $e$ is suitable at stage $s$ if $[0,k]\subseteq dom(\varphi_{f(e)})$ and $[0,2\mathrm{max}(\varphi_{f(e)}(A_{s+1})]\subseteq \mathrm{range}(\varphi_{f(e)})$. Notice that $\emptyset'$ can compute whether or not $e$ is suitable at stage $s$ uniformly in $e$ and $s$ because it can ask finitely many questions about convergence. Now let 
\[r_{s+1}=\mathrm{max}\{\varphi_{f(e)}^{-1}(i):e<s\text{ suitable at stage }s,\ i\leq 2\mathrm{max}(\varphi_{f(e)}(A_{s+1})\}+1.\]
Let $A=\bigcup_{s\in\omega} A_s$. By construction, $A\cap R_s\neq\emptyset$ because an element of $R_s$ was added at stage $s+1$. Now let $\pi=\varphi_{f(e)}$ be a computable permutation. Then $\pi$ is suitable at every stage because its domain and range are $\omega$. Now let $k$ be the element added at stage $s+2$ for some $s>e$. Then for every $i\leq 2\mathrm{max}(\pi(A_{s+1}))$,
\[k>r_{s+1}>\pi^{-1}(i).\]
Therefore $\pi(k)>2\mathrm{max}(\pi(A_{s+1}))$. Thus after finitely many elements, each element of $\pi(A)$ is more than double the previous element. It follows immediately that $\overline{\rho}(\pi(A))=0$. As $\pi$ was an arbitrary computable permutation, $A$ is intrinsically small.
\end{proof}

We can now show that there is an intrinsically small set which is not small for total computable functions.

\begin{theorem}
    There is a set of intrinsic density $0$ which is not small for total computable functions. That is, there is an intrinsically small set $A$ and a total computable function $f$ such that $\overline{\rho}(f(A))>0$.
\end{theorem}

\begin{proof}
    As defined by Jockusch and Schupp \cite{coarse}, let $R_e=\{n:2^e|n\text{ but }2^{e+1}\not|n\}$. Define $f:\omega\to\omega$ via $f(0)=0$ and $f(n)=e$, where $n\in R_e$. (Note that this is well-defined, as the $R_e$'s form a partition of $\omega\setminus\{0\}$.) $f$ is a total computable function. 
    \\
    \\
    By Lemma \ref{jumpstrategy}, there is an intrinsically small set $A$ such that $R_e\cap A\not=\emptyset$ for all $e$. Then $f(A)$ is cofinite (in fact it is either $\omega$ or $\omega\setminus\{0\}$), and therefore of intrinsic density $1$. (So $A$ catastrophically fails to have density $0$ under $f$.)
\end{proof}
We see from this example that the failure of injectivity allowed us to cast a wide net in search of elements of $A$ and then group them together to create a set of large density. Below, we shall see that we cannot even limit this to finite inverse images and preserve the property of being intrinsically small. In fact, we cannot even limit this to finite inverse images with uniformly computable size.
\\
\\
We shall need the notion of a hyperimmune set to do this. Recall that a disjoint strong array is a collection $\{D_{f(n)}\}_{n\in\omega}$ of finite sets coded by a total computable function $f$ and the canonical indexing of finite sets, where the $D_{f(n)}$'s are pairwise disjoint. A set $X$ is \textit{hyperimmune} if for every disjoint strong array $f$, there exists some $n$ with $D_{f(n)}\cap X=\emptyset$.
\begin{theorem}
\label{infinite}
    There is an intrinsically small set which is not small for the collection of all total computable functions $f$ such that $f^{-1}(\{n\})$ is finite (and uniformly computable) for all $n$. That is, there is an intrinsically small set $A$ and a total computable function $f$ such that $\overline{\rho}(f(A))>0$ and a total computable function $g$ such that $g(n)=|f^{-1}(\{n\})|$ for all $n$.
\end{theorem}

\begin{proof}
    Astor \cite{intrinsicsmallness} proved that the Turing degrees which contain an infinite intrinsically small set are those which are not weakly computably traceable. Kjos-Hanssen, Merkle, and Stephan \cite{highordnc} characterized these degrees as those which are High or DNC.
    \\
    \\
    It is well-known that there is a binary tree for which all paths are of PA degree. Recall that the PA degrees are exactly the $\text{DNC}_2$ degrees. Therefore, by the hyperimmune-free basis theorem, there is a $\text{DNC}_2$ degree that is hyperimmune-free. (For a review of this information, see Soare \cite{soare}.) This degree contains a set $A$ which is intrinsically small by the result of Astor. As $A$ is hyperimmune free, there exists a disjoint strong array $g$ such that $D_{g(n)}\cap A\neq\emptyset$ for all $n$. Without loss of generality, we can assume that $\mathrm{max}(D_{g(n)})<\mathrm{min}(D_{g(n+1)})$ for all $n$. (Given a disjoint strong array $g$, we can construct a new one $h$ as follows: $D_{h(0)}=D_{g(0)}$, and $D_{h(n+1)}$ is the first cell of the old array whose smallest element is larger than the largest element of $D_{h(n)}$.)
    \\
    \\
    Define $f:\omega\to\omega$ as follows: If $n\in D_{g(k)}$ for some $k$, let $f(n)=2k$. As $f$ is a disjoint strong array such that $\mathrm{max}(D_{g(n)})<\mathrm{min}(D_{g(n+1)})$, this is computable and well-defined. If $n\not\in\bigcup_{k\in\omega} D_{g(k)}$, then let $f(n)$ be the least odd number not realised as $f(m)$ for some $m<n$. Therefore $f$ is a total computable function with $|f^{-1}(\{n\})|$ finite and uniformally computable. (If $n=2k+1$ is odd, then the inverse image is a singleton. If $n=2k$ is even, then $f^{-1}(\{2k\})=D_{g(k)}$.) Furthermore, as $D_{g(n)}\cap A\neq\emptyset$ for all $n$, $f(A)$ contains all even numbers. Therefore $\overline{\rho}(f(A))\geq\frac{1}{2}$.
\end{proof}
We see that it is much more difficult for a set to be small for non-injective classes of functions. However, both examples relied heavily upon the fact that the functions were not injective. By switching our focus to (mostly) injective classes of functions, we can describe some classes of functions which any intrinsicall small set is small for. First, we provide an easy technical lemma.
\begin{lemma}
    Suppose $C$ is an infinite c.e. set. Then there exists an infinite, computable $H\subseteq C$ with $\rho(H)=0$.
\end{lemma}

\begin{proof}
    Let $\{c_i\}_{i\in\omega}$ be an enumeration of $C$. Then let $\{h_i\}_{i\in\omega}$ be such that $h_0=c_0$ and given $h_n$, $h_{n+1}=c_j$, where $j$ is the least index with $c_j>h_n+2^n$. Then $H$ is computable because it is a c.e. set with an increasing enumeration, and it clearly has density $0$.
\end{proof}

\begin{theorem}
\label{injective}
    Suppose that $A$ is an intrinsically small set. Then $A$ is small for the class of total computable injective functions with computable range.
\end{theorem}

\begin{proof}
    We argue by contrapositive: Suppose $f$ is total computable injective function with computable range, and $A$ is a set with $\overline{\rho}(f(A))>0$. Then we construct a computable permutation $\pi$ such that $\overline{\rho}(\pi(A))>0$.
    \\
    \\
    Let $H\subseteq \mathrm{range}(f)$ be a computable set of density $0$. Now define $\pi:\omega\to\omega$ as follows: If $f(n)\not\in H$, $\pi(n)=f(n)$. If $f(n)\in H$, let $\pi(n)$ be the least element of $H\cup\overline{\mathrm{range}(f)}$ not realized in the range of $\pi$ by $m<n$. Then $\pi$ is a computable permutation, and
    \[\rho_n(\pi(A))=\frac{|\pi(A)\upharpoonright n|}{n}\geq \frac{|f(A)\upharpoonright n|-|H\upharpoonright n|}{n}=\rho_n(f(A))-\rho_n(H).\]
    (The inequality comes from the fact that $\pi$ and $f$ agree on $f^{-1}(\mathrm{range}(f)\setminus H)$.) Therefore, we obtain
    \[\overline{\rho}(\pi(A))\geq \overline{\rho}(f(A))-\overline{\rho}(H)=\overline{\rho}(f(a))>0.\]
    Therefore $\pi$ is a computable permutation for which $\overline{\rho}(\pi(A))>0$, so $A$ is not intrinsically small.
\end{proof}

Note that simpler proofs of Theorem 2.5 exist which do not require us to create an error set and construct a permutation, however this proof is illustrative of the techniques we shall use for more difficult proofs.

\begin{corollary}
\label{image}
    If $A$ is intrinsically small and $f$ is a total computable injective function with computable range, then $f(A)$ is intrinsically small.
\end{corollary}

\begin{proof}
    This follows from Theorem \ref{injective} by the fact that $\pi(f(A))=\pi\circ f(A)$ and $\pi\circ f$ is a total computable injective function with computable range because $f$ is.
\end{proof}

\begin{corollary}
\label{join}
    If $A$ and $B$ are intrinsically small, then so is $A\oplus B$.
\end{corollary}

\begin{proof}
    If $f$ is the function sending $n$ to $2n$, and $g$ is the function sending $n$ to $2n+1$, then by Corollary \ref{image} $f(A)$ and $g(B)$ are both intrinsically small. It is easy to check that the union of two intrinsically small sets is intrinsically small, as the permutation of the union is the union of the images under the permutation. Therefore, $A\oplus B=f(A)\cup g(B)$ is intrinsically small.
\end{proof}

We can improve this result. The use of $H$ in the proof allows us to notice that we can change a subset of density $0$ in the range and not suffer any consequences for preserving intrinsic smallness.

\begin{definition}
    A (partial) function $f:\omega\to\omega$ is *-injective, or almost injective, if $\rho(\{n:|f^{-1}(\{n\})|>1\})=0$. That is, a (partial) function is almost injective if the subset of the range where injectivity fails has density $0$.
\end{definition}

\begin{theorem}
\label{almostinjective}
    Suppose that $A$ is an intrinsically small set. Then $A$ is small for the class of total computable *-injective functions with computable range.
\end{theorem}

\begin{proof}
    We again argue by contrapositive: Suppose $f$ is total computable *-injective function with computable range, and $A$ is a set with $\overline{\rho}(f(A))>0$. Then we construct a total computable injective function $g$ such that $\overline{\rho}(g(A))>0$ and invoke Theorem \ref{injective}.
    \\
    \\
    Let $H\subseteq \mathrm{range}(f)$ be infinite, computable, and have density $0$. Then define $g(n)=f(n)$ if $f(n)$ has not been realized in $\mathrm{range}(g)$ by some $m<n$, and to be the least element of $H$ not realized in $\mathrm{range}(g)$ otherwise. Then $g$ is injective, as $g(n)$ cannot be in $\mathrm{range}(g\upharpoonright n)$ for any $n$ by construction. Furthermore,
    \[\rho_n(g(A))=\frac{|g(A)\upharpoonright n|}{n}\geq \frac{|f(A)\upharpoonright n|-|H\upharpoonright n|-|\{k:|f^{-1}(\{k\})|>1\}\upharpoonright n|}{n}=\]
    \[\rho_n(f(A))-\rho_n(H)-\rho_n(\{k:|f^{-1}(\{k\})|>1)).\]
    This gives
    \[\overline{\rho}(g(A))\geq \overline{\rho}(f(A))-\overline{\rho}(H)-\overline{\rho}(\{k:|f^{-1}(k)|>1\}=\overline{\rho}(f(A))>0.\qedhere\]
\end{proof}

\begin{remark}
While an intrinsically small set is small for the class of total computable *-injective functions with computable range, the image under such functions is not intrinsically small: Take the set $A$ and function $f$ from the proof of Theorem \ref{infinite} and let $g(n)=2^{f(n)}$. Then $g$ is *-injective because its entire image has density zero. However, there is a computable permutation $\pi$ that maps $image(g)$ to the non-factorials and the complement to the factorials. Then $\pi\circ g(A)$ is all but finitely many of the non-factorials and is therefore density one.
\end{remark}

To this point, we've seen that injectivity almost everywhere has been essential in allowing all intrinsically small sets to be small for our class of functions. However, up to this point we've also relied heavily on knowing that the range is computable: if the range is not computable, we may potentially fill in part of the range that $A$ would have been sent to later. In this case, we'd need to shift where the elements of $A$ are sent, potentially sending the density to $0$ in the process. As we'll see below, there are cases in which we can avoid this issue.

\begin{theorem}
\label{density}
    Suppose $A$ is a set and $f$ is a *-injective function with $\overline{\rho}(f(A))=q>0$ and $\overline{\rho}(\mathrm{range}(f))-\underline{\rho}(\mathrm{range}(f))<q$. Then there is a *-injective function $g$ with computable range such that $\overline{\rho}(g(A))>0$.
\end{theorem}

\begin{proof}
    As $\mathrm{range}(f)$ is c.e., there is a computable subset $H$ of $\mathrm{range}(f)$ with $\underline{\rho}(H)>\overline{\rho}(\mathrm{range}(f))-q$ by Downey, Jockusch, and Schupp \cite{DJS}. In particular,
    \[\overline{\rho}(\mathrm{range}(f)\setminus H)\leq \overline{\rho}(\mathrm{range}(f))-\underline{\rho}(H)<q.\]
    Define $g:\omega\to\omega$ via $g(n)=f(n)$ if $f(n)\in H$, and $g(n)=0$ otherwise. Notice that $g$ is *-injective, as 
    \[\{n:|g^{-1}(\{n\})|>1\}\subseteq\{n:|f^{-1}(\{n\}|>1\}\cup\{0\}.\]
    Furthermore, $\mathrm{range}(g)=H\cup \{0\}$ is computable. Lastly, notice that
    \[\rho_n(g(A))=\frac{|g(A)\upharpoonright n|}{n}\geq\frac{|f(A)\upharpoonright n|-|\{k<n:k\not\in H\text{ and }k\in f(A)|}{n}\geq\]
    \[\frac{|f(A)\upharpoonright n|-|(\mathrm{range}(f)\setminus H)\upharpoonright n|}{n}=\rho_n(f(A))-\rho_n(\mathrm{range}(f)\setminus H).\]
    By the above fact that $\overline{\rho}(\mathrm{range}(f)\setminus H)\leq\overline{\rho}(\mathrm{range}(f))-\underline{\rho}(H)<q$, 
    \[\overline{\rho}(g(A))> \overline{\rho}(f(A))-q=q-q=0\]
    That is, $\overline{\rho}(g(A))>0$.
\end{proof}

\begin{corollary}
\label{hasdensity}
    Suppose that $A$ is an intrinsically small set. Then $A$ is small for the class of total computable *-injective functions whose range has defined density.
\end{corollary}

\begin{proof}
    We again argue by contrapositive: Suppose $f$ is total computable *-injective function whose range has defined density, and $A$ is a set with $\overline{\rho}(f(A))>0$. Then by Theorem \ref{density}, as $\overline{\rho}(\mathrm{range}(f))-\underline{\rho}(\mathrm{range}(f))=0$, there is a *-injective function $g$ with computable range such that $\overline{\rho}(g(A))>0$. The result follows by Theorem \ref{almostinjective}.
\end{proof}

By the remark following the proof of Theorem \ref{almostinjective}, we see that the image of an intrinsically small set under a total computable *-injective function whose range has defined density need not be intrinsically small. However if we restrict ourselves to injective functions, can we recover the analogue of Corollary \ref{image}? The same argument does not work, as the image of a c.e. set with defined density under a computable permutation need not have defined density.

\begin{question}
\label{preserve}
    If $A$ is intrinsically small and $f$ is a total computable injective function whose range has defined density, then is $f(A)$ intrinsically small?
\end{question}

Additionally, the natural follow-up question to Corollary \ref{hasdensity} remains open. This question is closely related to Question \ref{preserve}.

\begin{question}
\label{question-injective}
        Suppose that $A$ is an intrinsically small set. Is $A$ small for the class of total computable *-injective functions? Total computable injective functions?
\end{question}
Notice that if the answer here is yes, then the analogue of Corollary {\ref{image}} for computable injective functions follows immediately from the same argument. Therefore a positive answer yields a positive answer to Question {\ref{preserve}}, and a negative answer to Question {\ref{preserve}} yields a negative answer to Question {\ref{question-injective}}. The opposite direction also seems closely related, but any implications are not immediately obvious. 
\\
\\
Theorems \ref{almostinjective} and \ref{density} help to characterize what must happen in the scenario where the answer to Question \ref{question-injective} is no: The upper and lower density of the range are relatively far apart, allowing small elements of $f(A)$ to show up at late stages after any computable process ``thinks'' $\mathrm{range}(f)$ is done enumerating small elements.
\\
\\
Corollary {\ref{hasdensity}} can already be used in conjunction with known results. For example, Jockusch (correspondence with Astor) showed that r-maximal sets have intrinsic density (and therefore density) $1$, so the image of any intrinsically small set under a computable injective function whose range is maximal is small.

\section{Hyperimmunity and Intrinsic Smallness}

It is important to note that when studying whether or not certain properties relate to intrinsic smallness, we shall study the sets themselves rather than their degrees: coding tricks can show that every Turing degree contains a set with undefined density. In the c.e. degrees, this set can be taken to be c.e.

\begin{lemma}
\label{degree}
    Every Turing degree contains a set $W$ with $\underline{\rho}(W)=0$ and $\overline{\rho}(W)=1$. 
\end{lemma}

\begin{proof}
    Given $C$, let $D=\{n!:n\in C\}$ and $W=D\cup\bigcup_{n\in\omega}((2n)!,(2n+1)!)$. Then $W\equiv_T D\equiv_T C$, and $\underline{\rho}(W)=0$ because \[\rho_{(2n+2)!}(W)=\frac{|W\upharpoonright (2n+2)!|}{(2n+2)!}\leq \frac{(2n+1)!}{(2n+2)!}=\frac{1}{2n+2}.\]
    Conversely, $\overline{\rho}(W)=1$ as
    \[\rho_{(2n+1)!}(W)=\frac{|W\upharpoonright (2n+1)!|}{(2n+1)!}\geq \frac{(2n+1)!-(2n)!}{(2n+1)!}=1-\frac{1}{2n+1}.\]
    Clearly if $C$ is c.e., then so is $W$.
\end{proof}
We shall see below that additional properties on the starting set $C$ can be recovered in $W$ by modifying the construction.
\\
\\
We now turn our attention to hyperimmune sets, a competing notion of smallness. Astor \cite{intrinsicdensity} studied the connection between varying notions of immunity and intrinsic density thoroughly. In particular, it is known that hyperimmune sets have intrinsic lower density $0$, and therefore that hypersimple sets have intrinsic upper density $1$. (Hypersimple sets are c.e. sets whose complement is hyperimmune. Recall that hyperimmune sets are infinite by definition, so hypersimple sets are co-infinite.) One question left open in \cite{intrinsicdensity} (later answered by Astor in \cite{intrinsicsmallness} using a degree argument) was whether or not a hypersimple set could have lower density $0$, or at least non-$1$ lower density. The answer is yes, showing that hypersimple sets need not have defined density. We give a constructive proof, showing that every hypersimple set yields a Turing equivalent hypersimple set which has lower density $0$. (That is, every hypersimple set has an equivalent hypersimple set which is ``as small as possible.'')

\begin{theorem}
\label{hypersimple}
    Let $C$ be a hypersimple set. Then there is a hypersimple set $W\equiv_T C$ with $\underline{\rho}(W)=0$.
\end{theorem}

\begin{proof}
    As $C$ is hypersimple, it has intrinsic upper density (and therefore upper density) $1$. We cannot use the strategy from Lemma \ref{degree} directly, as the resulting set will not even be immune, let alone hyperimmune. To avoid this problem, we shall leave intervals of $C$ intact and introduce gaps between the intervals in noncomputable fashion. Informally, we first wish to shift portions of $C$ over to make large gaps, ensuring that the resulting set has lower density $0$. We then leave an even larger interval of $C$ intact (albeit shifted over finitely much) to ensure that the upper density is $1$. (See Figure \ref{hyperimmunefigure}.) Formally, we shall define c.e. sets $H_i$ and gaps $[u_i, u_i+m_i]$ inductively. Let $H_0=C$. Enumerate $H_0$ until there is a stage $s$ and a number $n$ such that we see $\rho_n(H_0)>\frac{1}{2}$, which exists because $C=H_0$ has upper density $1$. Then let $u_0=n$ and let $m_0$ be the least natural number such that $\frac{u_0}{u_0+m_0}<\frac{1}{2}$.
    \\
    \\
    Given $H_{e}$ and $[u_e,u_e+m_e]$, define $H_{e+1}$ and $[u_{e+1},u_{e+1}+m_{e+1}]$ as follows: Define $H_{e+1}=(H_e\upharpoonright u_e)\cup (H_e^{\geq u_e}+m_e)$. (For convenience, here $X^{\geq k}$ denotes $\{n\in X: n\geq k\}$, and $X+m=\{n+m:n\in X\}$.) Enumerate $H_{e+1}$ until there is a stage $s$ and a number $n>u_e+m_e$ such that $\rho_n(H_{e+1,s})>1-\frac{1}{e+2}$. Then set $u_{e+1}=n$ and $m_{e+1}$ to be the least natural number such that $\frac{u_{e+1}}{u_{e+1}+m_{e+1}}<\frac{1}{e+2}$. Finally, let $H$ be the set with characteristic function $H(m)=\lim_{n\to\infty} H_n(m)$. Note, first off, that $\bigcup_{e\in\omega} [u_e,u_e+m_e]$ is a c.e. set with increasing enumeration, and hence computable. Furthermore, note that $H$ itself is c.e., as $\lim_{n\to\infty} H_n(m)=H_s(m)$ for any $s$ with $u_s>m$. $\underline{\rho}(H)=0$ as desired, as $\rho_{u_i+m_i}(H)<\frac{1}{i+2}$ for all $i$.
    \\
    \\
    $H$ itself will not work as the desired $W$: The complement contains the computable subset $\bigcup_{e\in\omega} [u_e,u_e+m_e]$, so it is not even immune, let alone hyperimmune. Therefore, let $W=H\cup\bigcup_{n\in C} [u_n,u_n+m_n]$: that is, enumerate the $n$-th gap into $W$ whenever $n$ enters $C$. Then $W$ is c.e., and we claim that it is hypersimple.
    \\
    \\
    Recall that the principal function $p_A:\omega\to A$ of a set $A=\{a_0<a_1<a_2<\dots\}$ is the function such that $p_A(n)=a_n$. Also recall that a set is hyperimmune if and only if its principal function is not computably bounded. Suppose that $\overline{W}$ is not hyperimmune. Then it is bounded by some total computable function $f$. However, the total computable function $g$ defined via $g(n)=f(n+\mathop{\Sigma}_{i\leq n} m_i)$ must bound $\overline{C}$: The elements of $\overline{W}$ are the elements of $\overline{C}$ shifted up along with the corresponding gaps. The $n$-th element of $\overline{C}$ is smaller than the $n$-th non-gap element of $\overline{W}$ (as the $n$-th non-gap element of $\overline{W}$ is the $n$-th element of $\overline{C}$ shifted up by the gaps below it), which is at most the $n+\mathop{\Sigma}_{i\leq n} m_i$-th element of $\overline{M}$ because a gap in $\overline{W}$ corresponds to an element of $\overline{C}$ below the gap.
    \\
    \\
    Thus we have shown that $W$ is a hypersimple set. It is Turing equivalent to $C$ because $\bigcup_{e\in\omega} [u_e,u_e+m_e]$ is computable: $W$ can compute $C$ by ignoring the intervals, and $C$ can clearly compute $H$ and hence $W$.
\end{proof}

\begin{figure}
\begin{center}
\begin{tikzpicture}
    \draw[thin] (-1,0) -- (-1,0) node[anchor=north]{$0$} node[anchor=south]{$H_0^{\geq 0}+0$};
    \draw[ultra thick] (-1,0) -- (10,0) node[anchor=west]{$C=H_0$};
    
    \draw[ultra thick] (-1,-2) -- (1,-2) node[anchor=north]{$u_0$}
    node[anchor=south]{$H_0\upharpoonright u_0$};
    \draw[thin] (1,-2) -- (3,-2) node[anchor=north]{$u_0+m_0$}
    node[anchor=south]{$H_0^{\geq u_0}+m_0$};
    \draw[ultra thick] (3,-2) -- (10,-2) node[anchor=west]{$H_1$};
    
    \draw[ultra thick] (-1,-4) -- (1,-4);
    \draw[thin] (1,-4) -- (3,-4);
    \draw[ultra thick] (3,-4) -- (5.5,-4) node[anchor=north]{$u_1$}
    node[anchor=south]{$H_1\upharpoonright u_1$};
    \draw[thin] (5.5,-4) -- (9,-4)
    node[anchor=north]{$u_1+m_1$} node[anchor=south]{$H_1^{\geq u_1}+m_1$};
    \draw[ultra thick] (9,-4) -- (10,-4) node[anchor=west]{$H_2$};
    
    \node at (5, -5) {.};
    \node at (5,-5.1) {.};
    \node at (5,-5.22) {.};
    
    \node at (10.5, -5) {.};
    \node at (10.5,-5.1) {.};
    \node at (10.5,-5.22) {.};
    
\end{tikzpicture}
\caption{Visualization of the construction of $H$ in Theorem \ref{hypersimple}}
\label{hyperimmunefigure}
\end{center}
\end{figure}
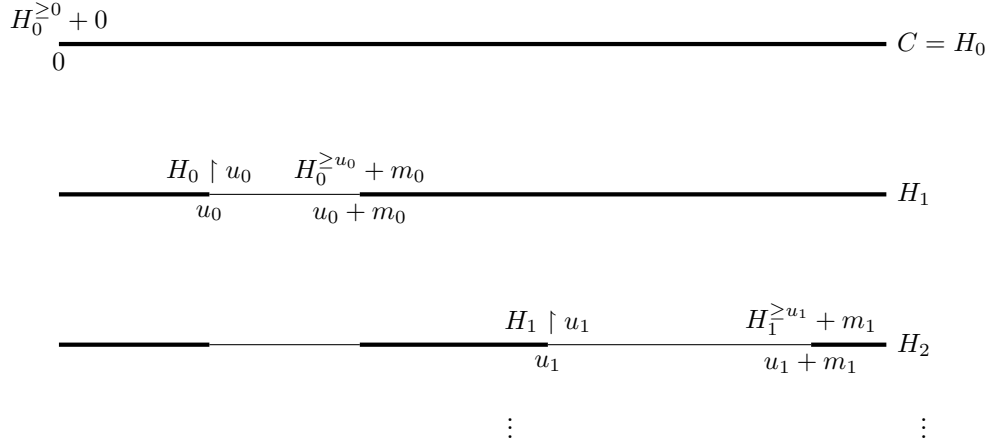

By using $C$ as an oracle rather than an enumeration of $C$, it is clear that this result also applies to co-hyperimmune sets in general, not just hypersimple sets.
\\
\\
Perhaps the most useful characterization of the hyperimmune sets is that a set is hyperimmune if and only if its principle function is not computably bounded. Recall that the principle function $p_X$ of an infinite set $X=\{x_0<x_1<x_2<x_3<\dots\}$ is the function such that $p_X(n)=x_n$. While Theorem \ref{hypersimple} shows that hyperimmunity and intrinsic smallness are unrelated notions of smallness, we would like to know whether it is possible to provide a simple characterization of intrinsic smallness using principal functions. Perhaps the most natural candidate is that of weak computable traceability from \cite{intrinsicsmallness}, which does provide us with a useful test for intrinsic smallness:

\begin{lemma}
\label{principalfunction}
    Suppose that $A$ is not intrinsically small. Then the principle function $p_A(n)$ of $A$ is weakly computably traced, i.e. there are computable functions $g$ and $h$ with $|D_{g(n)}|\leq h(n)$ for all $n$ and $p_A(n)\in D_{g(n)}$ for infinitely many $n$.
\end{lemma}

\begin{proof}
    As $A$ is not intrinsically small, there is a computable permutation $\pi$ such that $\overline{\rho}(\pi(A))=q>0$. Define functions $h=\lambda n(n!)$ and $g$ such that $D_{g(n)}=\pi^{-1}([0,n!))$. Then we claim that $g$ and $h$ witness that $p_A$ is weakly computably traced. 
    \\
    \\
    To get a contradiction, suppose this is not the case. Then $p_A(k)\in D_{g(k)}=\pi^{-1}([0,k!))$ for only finitely many $k$. In particular, $\pi(n)\geq n!$ for all but finitely many $n\in A$. This clearly implies that $\overline{\rho}(\pi(A))=0$, however, as $\rho_n(\pi(A))\leq \frac{s+m+1}{m!}$ where $s$ is the number of $k$ for which $p_A(k)\in \pi^{-1}([0,k!))$ and $m$ is the largest number with $m!\leq n$. As $\frac{s+m+1}{m!}$ approaches $0$ in the limit, this contradicts the fact that $\overline{\rho}(\pi(A))=q>0$, so $g$ and $h$ must witness that $p_A$ is weakly computably traced.
\end{proof}

The contrapositive of Lemma \ref{principalfunction} tells us that if the principle function of $A$ is not weakly computably traced, then $A$ is intrinsically small. Unfortunately, Theorem \ref{infinite} tells us that we cannot hope to reverse this in general. However, notice that the proof in fact proves a stronger statement: If $A$ is not intrinsically small, then it is weakly computably traced with witness $h=\lambda n(n!)$. That is, if $p_A$ is not weakly computably traced by $h$, then $A$ is intrinsically small. If this can be reversed, that would characterize the intrinsically small sets.

\begin{question}
Is it the case that if $A$ is intrinsically small, then $p_A$ is not weakly computably traced by $h=\lambda n(n!)$? If it is not the case, is there an intrinsically small set which does not dominate $h$? (I.e. $p_A(n)\leq n!$ infinitely often?)
\end{question}

Of course there are computably dominated intrinsically small sets by Theorem \ref{infinite}, however it is not clear if there are any ``nice'' computable functions (i.e. something naturally occurring in arithmetic or combinatorics) which dominate an intrinsically small set, or even which are not dominated by the principal function of one. Our usual strategy for constructing intrinsically small sets is no help, as it requires arbitrarily large witnesses.
\section{Relative Intrinsic Smallness}

The definition of intrinsic density, and by extension the definition of intrinsic smallness, admits a natural relativization:

\begin{definition}
\label{relativized}
The $X$-absolute upper density of $A\subseteq\omega$ is
\[\overline{P_X}(A)=\sup\{\overline{\rho}(\pi(A)):\pi\text{ an }X\text{ computable permutation}\}\]
and the absolute lower density of $A$ is
\[\underline{P_X}(A)=\inf\{\underline{\rho}(\pi(A)):\pi\text{ an }X\text{ computable permutation}\}.\]
If $\overline{P_X}(A)=\underline{P_X}(A)$, then we call this limit the $X$-intrinsic density of $A$ and denote it by $P_X(A)$.
\end{definition}

It is easy to see that no infinite, co-infinite set $A$ is $A$-intrinsically small, or in fact has $A$-intrinsic density. (One way to observe this is to note that the permutation taking $A$ to the set $W$ in $deg(A)$ from Lemma \ref{degree} is $A$-computable.) Furthermore, given a set $A$, the set of Turing degrees for which $A$ is not intrinsically small is closed upwards and contains the cone above $A$. One may ask if a set is intrinsically small, is it the case that this set is exactly the cone above $A$? The answer is no.

\begin{lemma}
    There is an intrinsically small set $A$ and a permutation $\pi\not\geq_T A$ such that $\overline{\rho}(\pi(A))>0$.
\end{lemma}

\begin{proof}
    Let $B$ and $C$ be Turing incomparable intrinsically small sets. (These exist given the result of Astor that the degrees containing intrinsically small sets are the degrees which are high or DNC.) Then by Corollary \ref{join}, $A=B\oplus C$ is intrinsically small. Now let $\pi$ be the $B$-computable permutation mapping $\{2n:n\in B\}$ to the non-factorials and the complement to the factorials. Then $\pi(B\oplus C)$ contains the non-factorials, and therefore has density $1$.
\end{proof}

As a corollary, we see that given an intrinsically small set $A$, the set of $X$ for which $A$ is $X$-intrinsically small need not be the degrees strictly below $A$: As $B$ and $C$ in the above proof are Turing incomparable, $B\oplus C$ is strictly Turing above $B$, but is not intrinsically small relative to $B$. However, it is clear that given a set $A$, the collection of Turing degrees of $X$ with $A$ $X$-intrinsically msall is closed downwards.. Must it be a Turing ideal? The following lemma shows the answer is no.

\begin{lemma}
    There is an intrinsically small set $A$ and sets $B,C$ with $A$ $B$-intrinsically small and $C$-intrinsically small but not $B\oplus C$-intrinsically small. That is, the set of $X$ for which $A$ is $X$-intrinsically small is not a Turing ideal.
\end{lemma}

\begin{proof}
    By the Sacks Splitting Theorem \cite{sacks}, there are low sets $B$ and $C$ such that $B\oplus C\equiv_T \emptyset'$. Therefore a modification of Lemma \ref{jumpstrategy} allows us to obtain a set $A\leq\emptyset'$ which is both $B$-intrinsically small and $C$-intrinsically small. (As $B$ and $C$ are low, $B'\equiv_T C'\equiv_T \emptyset'$, so $\emptyset'$ can enumerate the partial $B$ and $C$ computable injective functions and determine suitability for them.) However, $A$ cannot be $B\oplus C$-intrinsically small because $A\leq_T\emptyset'\equiv_T B\oplus C$.
\end{proof}

Note that although the set of $X$ for which $A$ is $X$-intrinsically small need not be a Turing ideal, Definition \ref{relativized} still makes sense if one considers all $\mathcal{I}$-computable permutations in a Turing ideal $\mathcal{I}$ rather than computable in a set $X$.
\\
\\
The following lemma allows us to describe the degrees of $X$-intrinsically small sets for certain $X$.

\begin{lemma}
\label{relativizedhighdnc}
    Let $X$ be an arithmetical set. Then the Turing degrees which contain an $X$-intrinsically small set $A$ are the $X$-high or $X$-DNC degrees.
\end{lemma}

\begin{proof}
    We merely need to check that the proof of Corollary 2.7 from Astor \cite{intrinsicsmallness} relativizes. It is straightforward to check that the proof given by Downey and Hirschfeldt \cite{downeyhirschfeldt} of the result of Kjos-Hanssen, Merkle, and Stephan \cite{highordnc} relativizes: a set $A$ is $X$-weakly computably traceable if and only if it is $X$-high or $X$-DNC. 
    \\
    \\
    Using this, the rest of the proof of \cite{intrinsicsmallness} Theorem 2.4 relativizes, and therefore \cite{intrinsicsmallness} Corollary 2.5 does as well. \cite{intrinsicsmallness} Theorem 2.6 also relativizes, which is straightforward to check. To obtain \cite{intrinsicsmallness} Corollary 2.7, Astor employs the following result of Jockusch \cite{jockusch}: Given some property $P$ of some sets of natural numbers, if there is an arithmetical set exhibiting $P$ and $P$ is closed under taking subsets, then the collection of Turing degrees which contain a set exhibiting $P$ is closed upwards. The relativized form of Lemma \ref{jumpstrategy} above yields an $X'$-computable $X$-intrinsically small set $A$. As $X$ is arithmetical, $A$ is arithmetical, so we may apply the result of Jockusch to obtain the relativized form of \cite{intrinsicsmallness} Corollary 2.7. 
\end{proof}

There is an obvious gap in Lemma \ref{relativizedhighdnc}. Specifically, can the arithmetical requirement on $X$ be dropped? There are certainly sets $X$ for which there are no arithmetical $X$-intrinsically small sets $A$: If $X=\emptyset^{(\omega)}$, then $X$ computes every arithmetical set and therefore there cannot be an arithmetical $X$-intrinsically small set. An important note here is that the relativization of \cite{intrinsicsmallness} Corollary 2.5 and Theorem 2.6 did not rely on the fact that $X$ was arithmetical, so we already know that $X$-weakly computably traced sets are not $X$-intrinsically small and that any non-$X$-weakly computably traced set computes an $X$-intrinsically small set for even non-arithmetical $X$.

\begin{question}
For which non-arithmetical sets $X$ are the degrees containing an $X$-intrinsically small set those which are $X$-high or $X$-DNC? For which non-arithmetical $X$ are they upwards closed?
\end{question}

A natural question arises from the appearance of $\emptyset^{(\omega)}$: We say a set $A$ is arithmetically intrinsically small if it is $X$-intrinsically small for every arithmetical set $X$. Is there an arithmetically intrinsically small set which is not $\emptyset^{(\omega)}$-intrinsically small?
It turns out that the answer is yes, as $\emptyset^{(\omega)}$ can uniformly compute all of the arithmetical permutations. Therefore a modification of Lemma \ref{jumpstrategy} allows us to construct a $\emptyset^{(\omega)}$-computable set which is arithmetically intrinsically small.

\section{Intrinsic Computability}

Having studied intrinsically small sets, we now turn our attention to their use as error sets in ``almost computable'' settings. Astor \cite{intrinsicdensity} first described four possible variations of ``intrinsic'' generic computability, that is ``intrinsic'' generic descriptions of $A$ which ensure the existence of generic descriptions of $\varphi_e(A)$ for all $e\in Perm$. The four notions differ by how uniformly we can obtain a generic description for a given permutation. We provide the generalizations of each of these notions to the remaining three notions of asymptotic computability mentioned in Section 1, which gives us a total of sixteen separate notions. Throughout this section $x$ will denote an arbitrary element of $\{\text{effective dense, generic, coarse, dense}\}$. We shall begin by describing the strongest of the four notions, which is the most overtly related to our study of intrinsically small sets.
\begin{definition}
$A\subseteq \omega$ is intrinsically $x$-ly computable if there is an $x$ description of $A$ with an intrinsically small error set.
\end{definition}

Astor originally defined this notion as strongly intrinsically $x$-ly computable, however we shorten the definition for the sake of readability.
\\
\\
This is the most natural intrinsic variant of asymptotic computability, as it is obtained by simply requiring the error set to meet a stronger smallness condition. As we shall see, the other three notions introduced in \cite{intrinsicsmallness} are not obtained by simply modifying the error set, but rather by introducing new restrictions on the computation.
\\
\\
We should verify that the intrinsically $x$-ly computable sets are not just the computable sets: clearly the computable sets meet this definition for any $x$, but are there noncomputable examples? It turns out that for the strongest notion, intrinsically effectively densely computable sets, this is not the case:

\begin{lemma}
    Suppose that $A$ is intrinsically effectively densely computable. Then $A$ is computable.
\end{lemma}

\begin{proof}
    By definition, if $A$ is intrinsically effectively densely computable, then the error set is an intrinsically small computable set. However, no infinite computable set can be intrinsically small, as there is a computable permutation that maps it to the nonfactorials and its complement to the factorials. Therefore, the error set must be finite. As $A$ differs from a computable set by only finitely much, it must be computable.
\end{proof}

Fortunately, the other three do admit noncomputable examples. For generic computability, as mentioned in \cite{intrinsicsmallness}, any c.e. set with intrinsic density $1$, such as a maximal set, is intrinsically generically computable. Similarly, any set of intrinsic density $1$ or $0$ is intrinsically coarsely computable. Notice that any intrinsically generically computable set with defined intrinsic density must have intrinsic density $0$ or $1$ and thus be intrinsically coarsely computable: Let $\varphi_e$ be an intrinsic generic description of $A$. If $\{n:\varphi_e(n)\downarrow=1\}$ is finite, then $A$ has intrinsic density $0$ because $A=\{n:\varphi_e(n)\downarrow=1\}\cup(A\cap\overline{W_e})$ is a union of a finite set with an intrinsically small set. If this set is not finite, then it is an infinite c.e. subset of $A$. Therefore the absolute upper density of $A$ is $1$ because every infinite c.e. set has a computable subset, which can be mapped to the nonfactorials by a computable permutation. As $A$ has defined intrinsic density and its absolute upper density is $1$, it must have intrinsic density $1$. In both cases, $A$ is intrinsically coarsely computable. The following lemma shows that the intrsincially generically computabile sets and the intrinsically coarsely computable sets are not the same, however.

\begin{lemma}
\label{coarseseparation}
    There is a intrinsically coarsely computable set which is not intrinsically generically computable.
\end{lemma}

\begin{proof}
     By Lemma \ref {jumpstrategy}, there is an intrinsically small set $A$ such that for each infinite c.e. set $W_e$ there exists $a_e\in A\cap W_e$ with $a_e<a_s$ for $e< s$. That is, there is a unique designated element $a_e$ of $A$ for each infinite c.e. set $W_e$. $\emptyset'$ cannot determine if a c.e. set is infinite, but it can ask if there is a large enough element of $W_e$ to continue the construction and put that into $A$ if it exists. This may designate some elements for finite c.e. sets, but this is acceptable.
     \\
     \\
     Now define $B\subseteq A$ by agreeing with $A$ away from the $a_e$'s and diagonalizing against the $e$-th turing machine using $B(a_e)$, i.e. $B(a_e)=1-\varphi_e(a_e)$. (Note that $\varphi_e(a_e)\downarrow$ because $a_e\in W_e$.) Then $B\subseteq A$ has intrinsic density $0$ and cannot be intrinsically generically computable because it disagrees with every turing machine with infinite domain at least once.
\end{proof}

The reverse separation remains open: it is easy to ensure that a given Turing function is not an intrinsic generic description by simply finding one place where it is wrong. However, to ensure that a given Turing function is not an intrinsic coarse description, we must force it to disagree on an infinite set which is not intrinsically small, which is more difficult. The natural strategy is to take an intrinsic generic description $W_i$, say a maximal set, and attempt to change it to diagonalize against the total functions in such a way that the description is still c.e. and its complement is still intrinsically small. The issue arises from our not being able to enumerate all of the total functions using computable indices: there is an enumeration of c.e. indices which contains exactly the computable sets (given an index $e$, enumerate $W_e$ so long as the enumeration is increasing, but do not enumerate smaller elements), but there is no way to distinguish the infinite sets from the finite ones. If we know a given c.e. index $e$ yields an infinite computable set, it is easy to wait for convergence of $\varphi_e$ and diagonalize against it on an infinite computable subset of $W_i$, forcing $\varphi_e$ to not be a n intrinsic coarse description. However if $W_e$ is in fact finite, then we will never see convergence, and failing to converge for the indices of finite sets will make the complement of our new enumeration no longer intrinsically small. If we give up waiting for convergence after some length of time, then there is no guarantee that an infinite computable set will ever enumerate quickly enough to be diagonalized against. 

\begin{question}
\label{genericseparation}
Is there an intrinsically generically computable set which is not intrinsically coarsely computable?
\end{question}

One potentially useful result for this question is the result of Arslanov \cite{arslanov} that the only c.e. DNC degree is $\emptyset'$. As mentioned above, we know from \cite{intrinsicsmallness} that the degrees which contain an intrinsically small set are those which are high or DNC. As the domain of an intrinsic generic description is c.e. and can compute an intrinsically small set (its complement), its degree must be high or DNC, and therefore high.
\\
\\
Fortunately, the answer to this question resolves the remaining implications involving intrinsically densely computable sets: 

\begin{lemma}
The intrinsically densely computable sets are exactly the intrinsically coarsely computable sets if every intrinsically generically computable set is intrinsically computable, and the intrinsically densely computable sets strictly contain all of the intrinsically generically computable sets and intrinsically coarsely computable sets if this is not the case.
\end{lemma}

\begin{proof}
    By Lemma \ref{coarseseparation} there is a set $B$ which is intrinsically coarsely computable but not intrinsically generically computable. Let $A$ be a set which is intrinsically generically computable but not intrinsically coarsely computable. An application of Corollary \ref{join} tells us that $A\oplus B$ is intrinsically densely computable, but it is clear that it cannot be intrinsically coarsely computable or intrinsically generically computable because any intrinsic coarse/generic description of $A\oplus B$ would necessarily yield an intrinsic coarse/generic description of $A$/$B$.
    \\
    \\
    Now suppose that every intrinsically generically computable set is intrinsically coarsely computable, and let $A$ be intrinsically densely computable with witness $\varphi_e$. Then the set $B$ defined via the characteristic function 
    \[\chi_B(n)=\begin{cases} 
      \varphi_e(n) & n\in W_e \\
      0 & n\in \overline{W_e}
   \end{cases}\]
   is intrinsically generically computable with witness $\varphi_e$. Therefore it is intrinsically coarsely computable via some total witness $\varphi_i$. Therefore $\varphi_i$ witnesses that $A$ is intrinsically coarsely computable as well because the error set is contained within the union of two intrinsically small sets (the complement of $W_e$ and the error set of $\varphi_i$ on $B$) and thus is intrinsically small.
\end{proof}

The remaining three generalizations of asymptotic computation to the intrinsic setting use a separate idea: Rather than having an intrinsically small error set that ensures the existence of descriptions, we simply assert that descriptions must exist for any computable permutation. Varying the level of uniformity for these descriptions is how we reach three separate notions (Recall that $x\in\{\text{effective dense, generic, coarse, dense}\}$):

\begin{definition}
\mbox{}
\begin{itemize}
    \item $A$ is weakly intrinsically $x$-ly computable if $\varphi_e(A)$ is $x$-ly computable for every $e\in Perm$.
    \item $A$ is uniformly $x$-ly computable if there is a computable function $f(e,n)$ such that $\lambda n(f(e,n))$ is a(n) $x$ description of $\varphi_e(A)$ when $e\in Perm$.
    \item $A\subseteq\omega$ is oracle $x$-ly computable if there is a Turing functional $\Phi_i$ such that $\Phi_i^X$ is a(n) $x$ description of $\varphi_e(A)$ whenever $e\in Perm$ and $X=\mathrm{graph}(\varphi_e)$.
\end{itemize}
\end{definition}

As in the case of the intrinsically $x$-computable sets, Astor's original definitions were ``uniformly intrinsically $x$-ly computable'' and ``oracle intrinsically $x$-ly computable,'' however we shorten these definitions for readability.
\\
\\
It is immediate that all of the straightforward implications from asymptotic computability apply here in each of the three cases, i.e. uniformly coarsely computable sets are uniformly densely computable and so on. Furthermore, it is easy to see that for all $x\in\{\text{effective dense, generic, coarse, dense}\}$, intrinsically $x$-ly computabile sets are uniformly and oracle $x$-ly computable, which both in turn are weakly $x$-ly computable.  Furthermore, albeit slightly less trivial, is the fact that oracle $x$-ly computable sets are uniformly $x$-ly computable: Given a Turing functional $\Phi_i$ which witnesses that $A$ is oracle $x$-ly computable, define the partial computable function $f(e,n)$ via $f(e,n)=\Phi_i^{\mathrm{graph}(\varphi_e)}(n)$. Then the definition of oracle $x$-ly computable ensures that this function $f$ witnesses uniformly $x$-ly computable. This means that for a fixed $x$, the four notions form a chain.
\\
\\
As noted in \cite{intrinsicdensity}, it is unclear at first if these notions are all distinct (i.e. whether or not the chain collapses), even when restricting ourselves just to the generic case. Below we shall see that they are not distinct here, although the argument will not generalize to the coarse and dense settings. However, a slight modification of it shall provide a similar but not identical result for the effective dense setting.

\begin{theorem}
\label{genericcollapse}
Suppose that $A$ is oracle generically computable. Then $A$ is intrinsically generically computable.
\end{theorem}

\begin{proof}
Let $\Phi_i$ witness that $A$ is oracle generically computable. Then define the partial computable function $f$ as follows: Note that the set of finite binary strings $\sigma$ which are initial segments of graphs of injective functions is computable. For $\sigma$ in this set, let $f_\sigma$ denote the partial injective function with finite range such that $\mathrm{graph}(f_\sigma)$ is the infinite binary string obtained by adding infinitely many $0$'s to $\sigma$. Compute $f(n)$ by searching for such a $\sigma$ with $n\in \mathrm{range}(f_\sigma)$ and $\Phi_i^{\sigma}(f_\sigma(n))\downarrow$. If one is found, define $f(n)=\Phi_i^{\sigma}(f_\sigma(n))$ for the first such $\sigma$. Otherwise, $f(n)\uparrow$.
\\
\\
First, note that $f(n)\downarrow$ implies $f(n)=A(n)$: If $f(n)\downarrow$, then there is some $\sigma$ such that $\Phi_i^{\sigma}(f_\sigma(n))\downarrow$. As $\sigma$ is an initial segment of the graph of an injective function, $\sigma$ can be extended to $X$ where $X$ is the graph of some computable permutation $\varphi_e$. Then as $\Phi_i$ witnesses that $A$ is oracle generically computable, $\Phi_i^X$ is a generic description of $\varphi_e(A)$, so $\Phi_i^X(\varphi_e(n))\downarrow$ implies
\[\Phi_i^X(\varphi_e(n))=\varphi_e(A)(\varphi_e(n))=A(n).\]
In particular, 
\[A(n)=\Phi_i^X(\varphi_e(n))=\Phi_i^\sigma(f_\sigma(n))=f(n)\]
by the finite use principle.
\\
\\
Therefore, it remains to show that the domain of $f$ has intrinsic density $1$. Notice that if $\varphi_e$ is a permutation, then $\varphi_e(dom(f))$ contains $dom(\Phi_i^{\mathrm{graph}(\varphi_e)})$, as if $\Phi_i^{\mathrm{graph}(\varphi_e)}(k)\downarrow$, there is an initial segment $\sigma$ of $\mathrm{graph}(\varphi_e)$ with $k\in \mathrm{range}(f_\sigma)$ that witnesses convergence, and therefore witnesses $f(\varphi_e^{-1}(k))\downarrow$. However, $\underline{\rho}(dom(\Phi_i^{\mathrm{graph}(\varphi_e)}))=1$ as $\Phi_i^{\mathrm{graph}(\varphi_e)}$ is a generic description of $\varphi_e(A)$ and therefore has density $1$. Thus $dom(f)$ has density $1$ under every computable permutation and thus has intrinsic density $1$ as desired.
\end{proof}

\begin{corollary}
\label{effectivedensecollapse}
    Suppose that $A$ is oracle effective densely computable. Then $A$ is intrinsically generically computable.
\end{corollary}

\begin{proof}
    Construct the description $f$ of $A$ as in the proof of Theorem \ref{genericcollapse}, however instead of searching for convergence, search for convergence to either $0$ or $1$.
\end{proof}

As mentioned above, this argument does not in general apply to oracle coarsely computable sets and oracle densely computable sets. The issue lies in the fact that coarse and dense computation allows for mistakes, so we cannot ensure that any convergent computation is correct. 
\\
\\
The remaining implications remain open other than the previously observed chains. The difficulty in separating these notions lies in the fact that the constructed sets cannot be described by building one error set, but rather have a different error set for each computable permutation. More importantly, these countably many computable requirements are heavily interlocked: Consider attempting to construct a weakly intrinsically generically computable set which is not weakly intrinsically coarsely computable. As an example, we may try to define an error set for the identity permutation. However, this defines the membership of the constructed set on a given c.e. set $W_e$. If we wish to diagonalize for a given computable permutation $\pi$, we may find that $\pi(W_e)$ has density $1$, in which case we can't respect $W_e$ and also diagonalize on a set of positive density.

\end{document}